\documentclass[british, 12pt, reqno]{amsart}
\usepackage[T1]{fontenc}
\usepackage[latin9]{inputenc}
\usepackage{amsmath, amsthm, amssymb, stmaryrd}
\usepackage{mathrsfs}
\usepackage{geometry}
\usepackage{amsmath}
\usepackage{amsthm}
\usepackage{amssymb}
\usepackage{xcolor}
\usepackage{dsfont}

\usepackage{cancel}
\usepackage{soul}
\usepackage{comment}
\usepackage{color}
\usepackage{amscd}
\usepackage{tabularx}
\usepackage{url}
\usepackage{eurosym}
\usepackage{braket}
\usepackage{hyperref}
\usepackage{enumerate}
\usepackage{graphicx}

\makeatletter
\def\widebreve{\mathpalette\wide@breve}
\def\wide@breve#1#2{\sbox\z@{$#1#2$}%
     \mathop{\vbox{\m@th\ialign{##\crcr
\kern0.08em\brevefill#1{0.8\wd\z@}\crcr\noalign{\nointerlineskip}%
                    $\hss#1#2\hss$\crcr}}}\limits}
\def\brevefill#1#2{$\m@th\sbox\tw@{$#1($}%
  \hss\resizebox{#2}{\wd\tw@}{\rotatebox[origin=c]{90}{\upshape(}}\hss$}
\makeatletter

\makeatletter
\numberwithin{equation}{section}
\numberwithin{figure}{section}
\theoremstyle{plain}
\newtheorem{thm}{\protect\theoremname}[section]
\theoremstyle{plain}
\newtheorem{cor}[thm]{\protect\corollaryname}
\ifx\proof\undefined
\newenvironment{proof}[1][\protect\proofname]{\par
	\normalfont\topsep6\p@\@plus6\p@\relax
	\trivlist
	\itemindent\parindent
	\item[\hskip\labelsep\scshape #1]\ignorespaces
}{%
	\endtrivlist\@endpefalse
}
\providecommand{\proofname}{Proof}
\fi
\theoremstyle{remark}

\theoremstyle{plain}
\newtheorem{lem}[thm]{\protect\lemmaname}

\newcommand{\Sam}[1]{{\color{red} \sf Sam: [#1]}}

\makeatother

\usepackage{babel}
\providecommand{\corollaryname}{Corollary}
\providecommand{\lemmaname}{Lemma}
\providecommand{\remarkname}{Remark}
\providecommand{\theoremname}{Theorem}

\numberwithin{equation}{section}
\numberwithin{figure}{section}
\theoremstyle{plain}
\newtheorem{remark}[thm]
{Remark}
\newtheorem{prop}[thm]{\protect Proposition}
\newtheorem{conj}[thm]{\protect Conjecture}

\newtheorem{qn}[thm]{Question}

\def\1{{\textcolor{red} {1}}}
\def\d1{{\textcolor{red} {d-1}}}

\def \bN {\mathbb N}

\def \bQ {\mathbb Q}
\def \bR {\mathbb R}

\def \bZ {\mathbb Z}

\def \sL {\mathscr L}

\def \rd {\mathrm d}

\def \ri {\mathrm i}

\def \ba {\mathbf a}

\def \bc {\mathbf c}
\def \bd {\mathbf t}

\def \bm {\mathbf m}

\def \bn {\mathbf n}
\def \bfN {\mathbf N}

\def \bs {\mathbf s}
\def \bt {\mathbf t}

\def \bv {\mathbf v}
\def \bx {\mathbf x}

\def \by {\mathbf y}

\def \bzero {{\boldsymbol 0}}
\def \bone {{\boldsymbol 1}}

\def \btau {\boldsymbol{\tau}}

\def \balp {{\boldsymbol{\alp}}}

\def \bgam {{\boldsymbol{\gam}}}
\def \blam {{\boldsymbol{\lambda}}}
\def \bdel {{\boldsymbol{\del}}}
\def \brho {{\boldsymbol{\rho}}}

\def \bomega {\boldsymbol{\omega}}
\def \bome {{\boldsymbol{\omega}}}

\def \fB {\mathfrak B}
\def \fC {\mathfrak C}
\def \fD {\mathfrak T}

\def \fT {\mathfrak T}

\def \cC {\mathcal C}

\def \cG {\mathcal G}

\def \cL {\mathcal L}
\def \cM {\mathcal M}

\def \cO {\mathcal O}
\def \cP {\mathcal P}
\def \cQ {\mathcal Q}
\def \cR {\mathcal R}

\def \cU {\mathcal U}

\def \cW {\mathcal W}

\def \le {\leqslant}
\def \leq {\leqslant}
\def \ge {\geqslant}
\def \geq {\geqslant}

\def \dim {\mathrm{dim}}
\def \dimh {{\mathrm{\dim_H}}}

\def \det {\mathrm{det}}

\def \vol {\mathrm{vol}}

\def \Bad {{\mathbf{Bad}}}
\def \Mad {{\mathbf{Mad}}}

\def \d {{\mathrm{d}}}

\def \ds1 {\mathds{1}}

\def \alp {{\alpha}}
\def \bet {{\beta}}
\def \gam {{\gamma}}
\def \del {{\delta}}
\def \eps {{\varepsilon}}
\def \kap {{\kappa}}
\def \lam {{\lambda}}
\def \ome {{\omega}}

\makeatother

\usepackage{babel}
\providecommand{\lemmaname}{Lemma}

\providecommand{\theoremname}{Theorem}
\usepackage{xifthen,ifthen}
\makeatletter
\newcounter{@ToDo}
\newcommand{\todo@helper}[1]{%
	({\color{blue}TODO~\arabic{@ToDo}: {#1\@addpunct{.}}})%
}
\newcommand{\todo}[1]{\stepcounter{@ToDo}%
	\relax\ifmmode\text{\todo@helper{#1}}%
	\else\todo@helper{#1}\fi%
}
\newcounter{@cdo}
\newcommand{\cdo@helper}[1]{%
	({\color{red}CITE~\arabic{@cdo}: {#1\@addpunct{.}}})%
}
\newcommand{\cdo}[1]{\stepcounter{@cdo}%
	\relax\ifmmode\text{\cdo@helper{#1}}%
	\else\cdo@helper{#1}\fi%
}
\newcommand{\mmod}[1]{\,\,\mathrm{mod}\,\,#1}

\def \ds1 {\mathds{1}}

\def \alp {{\alpha}}
\def \bet {{\beta}}
\def \gam {{\gamma}}
\def \del {{\delta}}
\def \eps {{\varepsilon}}
\def \epsilon {{\varepsilon}}
\def \kap {{\kappa}}
\def \lam {{\lambda}}
\def \ome {{\omega}}

\def \Lam {{\Lambda}}

\begin{document}

\author{Sam Chow \and Niclas Technau}
\address{Mathematics Institute, Zeeman Building, University of Warwick, Coventry CV4 7AL, United Kingdom}
\email{sam.chow@warwick.ac.uk}
\address{Max Planck Institute for Mathematics, Vivatsgasse 7, 53111 Bonn, Germany; and
Mathematical Institute, University of Bonn, Endenicher Allee 60, 53115, Bonn, Germany}
\email{ntechnau@uni-bonn.de; technau@mpim-bonn.mpg.de}

\title{Smooth discrepancy and Littlewood's conjecture}

\begin{abstract}
Given $\balp \in [0,1]^d$, we estimate
the smooth discrepancy of the Kronecker sequence
$(n \balp \mmod 1)_{n\geq 1}$. 
We find that it can be smaller than the classical discrepancy of 
\textbf{any} sequence when $d \le 2$, and can even be bounded in the case $d=1$. To achieve this, 
we establish a novel deterministic 
analogue of Beck's local-to-global principle (Ann. of Math. 1994), 
which relates the discrepancy of a Kronecker sequence 
to multiplicative diophantine approximation.
This opens up a new avenue of attack for Littlewood's conjecture.
\end{abstract}

\maketitle

\section{Introduction}

\subsection{Badly and madly approximable numbers}

It follows from Dirichlet's approximation theorem that if $\alp \in \bR$ then
\[
n \| n \alp \| < 1
\]
for infinitely many $n \in \bN$, where 
$\| x \| = \min \{ |x - a|: a\in \bZ \}$. 
Moreover, Khintchine's theorem implies that
\[
\liminf_{n \to \infty} n \| n \alp \| = 0
\]
for Lebesgue almost all $\alp$. However, this fails for a set $\Bad$ of \emph{badly approximable} numbers, which has Hausdorff dimension 1. Badly approximable numbers are ubiquitous in diophantine approximation, and are characterised by being irrational and having bounded partial quotients in the continued fraction expansion.

\emph{Multiplicative diophantine approximation} is about approximating several numbers at once with the same denominator, and multiplying the errors. For two numbers, a natural analogue of a badly approximable number is a pair $(\alp, \bet)$ such that
\[
\liminf_{n \to \infty} n \| n \alp \| \cdot \| n \bet \| > 0.
\]
Famously, Littlewood's conjecture asserts that there are no such pairs.

One can try to quantify matters further, and experts in the field will be familiar with the `mad conjectures' of Badziahin and Velani \cite{BV2011}, who defined
\[
\Mad^\lam = \left \{
(\alp, \bet) \in \bR^2: \liminf_{n \to \infty} n (\log n)^\lam
\| n \alp \| \cdot \| n \bet \| > 0 \right \}
\]
for $\lam \ge 0$. They suggested that $\Mad^\lam$ is empty whenever $\lam < 1$. Note that this is a strong form of Littlewood's conjecture. On the other hand, Badziahin \cite{Bad2013} showed that
\[
\dimh(\Mad^\lam) = 2
\qquad (\lam > 1).
\]
This extra logarithm also shows up in the work of Peck \cite{Pec1961}, who strengthened a result of Cassels and Swinnerton-Dyer to show that if $\alp, \bet$ lie in the same cubic field then
\[
\liminf_{n \to \infty} n (\log n) \| n \alp \| \cdot \| n \bet \| < \infty.
\]

Badziahin's result is somewhat stronger than what is stated above. We define the multiplicative height of $\bm \in \bR^{d}$
by
\[
H(\bm) = \prod_{i \le d} 
\max \{ 1, |m_i| \}.
\]

\begin{thm}
[Badziahin \cite{Bad2013}]
\label{BadziahinResult}
Let $f(n) = (\log n) \log \log n$. Then
\[
\dim_{\mathrm{H}}
(\Mad(f) \cap \Mad^*(f)) = 2,
\]
where
\begin{align*}
\Mad(f) &= \{ (\alp, \bet) \in \bR^2: \liminf_{n \to \infty}
n f(n) \| n \alp \| \cdot \| n \bet \| > 0 \},
\\
\Mad^*(f) &=
\{ 
(\alp, \bet) \in \bR^2: \liminf_{\bn \in \bZ^2}
H(\bn) f(H(\bn))
\| \bn \cdot (\alp, \bet) \| > 0 \}.
\end{align*}
\end{thm}
A result of Gallagher implies 
that almost all pairs $(\alp, \bet)$ 
satisfy Littlewood's conjecture 
and even lie in the complement of $\Mad^2$. Because 
of the one-periodicity of $\| \cdot\|$, it is 
natural to restrict attention to 
$(\alp, \bet)\in [0,1)^2$.
The full statement of Gallagher's theorem
has the quintessential flavour 
of a result in metric diophantine approximation.

\begin{thm} 
[Gallagher \cite{Gal1962}]
\label{thm: Gallagher}
Let $\psi: \bN \to [0,\infty)$ be non-increasing. Then
the set of
$(\alp, \bet) \in [0,1)^2$ such that
\begin{equation}
\label{GallagherIneq}
\| n \alp \| \cdot \| n \bet \| < \psi(n)
\end{equation}
holds for infinitely many $n \in \bN$ has Lebesgue measure
\[
\begin{cases}
0, &\text{if } 
\displaystyle \sum_{n=1}^\infty \psi(n) \log n < \infty, \\ \\
1, &\text{if } 
\displaystyle \sum_{n=1}^\infty \psi(n) \log n = \infty.
\end{cases}
\]
\end{thm}

The motivation for the threshold comes from probability theory. 
The convergence part of Theorem \ref{thm: Gallagher}
follows directly from the first Borel--Cantelli lemma.
In contrast, the divergence part
is the substance of the theorem. Since
the events \eqref{GallagherIneq} are 
far from being pairwise independent,
it is non-trivial that they satisfy 
the conclusion of the second Borel--Cantelli lemma.

What about specific values of $\alp$ and $\bet$? The inequality \eqref{GallagherIneq} holds if 
\begin{equation*} \label{2Kron}
(n \alp \mmod 1, n \bet \mmod 1)
\end{equation*}
lies in a small, axis-parallel rectangle around the origin. These points form a \emph{Kronecker sequence}. Their distribution
\begin{enumerate}[(i)]
\item near the origin (locally), and
\item throughout the torus $(\bR/\bZ)^2$ (globally)
\end{enumerate}
are inextricably tied. To elaborate, 
we recall some uniform distribution theory.

\subsection{Low-discrepancy sequences}

The \emph{discrepancy} of the first $N$ points of a sequence $(\bx_n)_{n \ge 1}$ in $[0,1)^d$ is
\[
D_N := \sup_\cR 
|\# \{ n \le N: \bx_n \in \cR \} - \lam_d(\cR) N|,
\]
where $\cR$ ranges over axis-parallel hyper-rectangles in $[0,1)^d$, and where $\lam_d$ is $d$-dimensional Lebesgue measure.
Up to a bounded multiplicative factor, this equals the \emph{star discrepancy} $D_N^*$, wherein the boxes are restricted to have the origin as a vertex \cite{KN1974}. A fundamental result in discrepancy theory, from 1945, is
due to van Aardenne-Ehrenfest. Her result is as follows.
\begin{thm}
[van Aardenne-Ehrenfest \cite{Aar1945}]
\label{thm: Ehrenfest}
For any sequence, the discrepancy $D_N$ is unbounded.
\end{thm}

Theorem~\ref{thm: Ehrenfest} 
was famously strengthened by Roth \cite{Rot1954}, 
in what he considered to be his best work~\cite{Chen2015}. 
Further important refinements are due to Schmidt \cite{Sch1972} and
Bilyk--Lacey--Vagharshakyan \cite{BLV2008}, 
as discussed in Bilyk's survey article \cite{Bil2011}. 
We write $f \ll g$ or $g \gg f$ if $|f| \le C|g|$ pointwise, for some $C$, and we write $f \asymp g$ if $f \ll g \ll f$. In the case $d=1$, Schmidt \cite{Sch1972} established the optimal result that
\[
D_N \gg \log N
\]
holds for infinitely many $N \in \bN$. Upon coupling \cite[Theorem 1.2]{Bil2011} with the discussion after Theorem 1.1 therein, we see that if $d \ge 2$ then for some $\eta = \eta_d > 0$ and any sequence in $[0,1)^d$, there are infinitely many $N \in \bN$ such that
\begin{equation}
\label{GeneralLower}
D_N \gg (\log N)^{\frac{d}{2} + \eta}.
\end{equation}

\subsection{A local-to-global principle}

To illustrate the idea, let us start with the case $d=1$. Here, the analogue of a counterexample to Littlewood's conjecture is a badly approximable number. One can ask what a lower bound for $\|n \alp \|$ would imply about the global distribution of the Kronecker sequence $(n \alp \mmod 1)_{n \ge 1}$. Khintchine showed that if $f: (0,\infty) \to (0, \infty)$ is increasing then, for almost all $\alp \in \bR$, the discrepancy of the corresponding Kronecker sequence satisfies
\[
\| n \alp \| \gg \frac1{nf(\log n)}
\quad
\Leftrightarrow
\quad
D_N \ll (\log N) f(\log \log N)
\quad
\Leftrightarrow
\quad
\sum_{n=1}^\infty \frac1{f(n)} < \infty.
\]
Turning now to higher dimensions, 
one might hope to profitably link 
the multiplicative approximation properties of $\balp$ 
to the discrepancy of its Kronecker sequence. 
In 1994, Jozsef Beck used this philosophy 
to establish the following `metric' bound. To avoid pathologies, 
we interpret $\log x$ as $\max \{1, \log x\}$ throughout.

\begin{thm} [Beck \cite{Bec1994}]
Let $f: (0,\infty) \to (0,\infty)$ be increasing. Then,
for almost all $\balp \in \bR^d$, the discrepancy of the corresponding Kronecker sequence satisfies
\[
D_N \ll (\log N)^d f(\log \log N) \Leftrightarrow \sum_{n=1}^\infty \frac1{f(n)} < \infty.
\]
\end{thm}

\begin{cor} If $\eps > 0$ then, for almost all $\balp \in \bR^d$, the discrepancy of the corresponding Kronecker sequence satisfies
\[
D_N \ll_\eps (\log N)^d (\log \log N)^{1+\eps}.
\]
\end{cor}

\noindent
As Beck explains in the introduction of \cite{Bec1994}, it is believed that $(\log N)^d$ should be the optimal bound.

The metric nature of the statement 
is caused by some parts of the argument requiring averaging over $\balp$. 
The threshold comes from the \emph{dual multiplicative approximation} rate that applies to almost every $\balp \in \bR^d$, and we state \cite[Lemma 4.3]{Bec1994} below.

\begin{lem} 
Let $f: (0,\infty) \to (0,\infty)$ be increasing with 
\[
\sum_{n=1}^\infty \frac1{f(n)} = \infty.
\]
Then, for almost every $\balp \in \bR^d$, there are infinitely many $\bn \in \bN^d$ such that
\[
n_1 \cdots n_d (\log(n_1 \cdots n_d))^d f(\log \log (n_1^2 \cdots n_d^2)) \| \bn \cdot \balp \| < 1.
\]
\end{lem}

This begs, once again, the question of whether one can relate multiplicative diophantine properties of $\balp$ to the discrepancy of the corresponding Kronecker sequence, for all $\balp$. In the case of classical discrepancy, this is too much to hope to have in any meaningful way, even in dimension $d=1$. Indeed, we will see that the classical analogue of our first new result --- Theorem \ref{thm: main} below --- is false, because it would produce impossibly small discrepancies in the cases $d=1$ and $d=2$.
In Beck's proof, the roadblock appears to be the use of a sharp cutoff in defining the discrepancy. The purpose of the present article is to resolve this matter.

\subsection{Smooth discrepancy}

The idea of going beyond indicator functions in this context features prominently in numerical integration, and is called \emph{quadrature} when $d = 1$ and \emph{cubature} when $d \ge 2$. We refer the reader to Temlyakov's article~\cite{Tem2003} for further discussion.

As well as the $d$ dimensions of the box, there is a further dimension for the term of the sequence. We smooth in all 
\[
k = d + 1
\]
dimensions.
Let 
$\bomega =(\omega_1,\ldots,\omega_k):\bR^k\rightarrow\bR^k$
be such that
\begin{equation}
\label{eq omega supported dyadically}
\ome_i: \bR \rightarrow [0, \infty), \quad 
\mathrm{supp}(\omega_i) \subseteq [-2,2],
\quad \omega_i \in C^\ell, 
\end{equation}
where $C^\ell$ denotes the set of the $\ell$-times
continuously differentiable functions, and
\begin{equation}
\label{positive}
\widehat \omega_i \ge 0, \qquad
\widehat \ome_i(0) > 0
\end{equation}
for each $i \le k$. Here we recall that the Fourier transform is given by
$$
\widehat \ome_i (\xi) := \int_\bR \ome_i(x) e(-\xi x)\rd x,
\quad \mathrm{where}\,\,e(z) = e^{2\pi \ri z}.
$$ 
Denote by $\cG_{k,\ell}$ the set of $\bome: \bR^k \to \bR^k$ 
satisfying \eqref{eq omega supported dyadically} and \eqref{positive}
for each $i \le k$.

For $\bgam \in \bR^d$
and $\brho \in (0,1/2)^d$, we define test boxes $\fB$
and their volume via
\begin{equation}
\label{def: box B}
\fB = (\bgam; \brho),
\qquad \vol(\fB; N) = \rho_1 \cdots \rho_d N.
\end{equation}
The smooth discrepancy in $\fB$ is the smooth point count 
minus the expected count:
\[
D_{\bomega}(\fB; N) := 
\sum_{\ba \in \bZ^k}
\ome_k \left(\frac{a_k}{N}\right) 
\prod_{i \le d}
\ome_i \left(
\frac{a_i + a_k \alp_i - \gam_i}{\rho_i}\right) -
\fC(\bomega) \vol(\fB; N),
\]
where
\begin{equation}
\label{def expect weight}
\fC(\bomega)
= \widehat{\omega}_1(0)
\cdots
\widehat{\omega}_k(0).
\end{equation}
We associate to $\bome \in \cG_{k,\ell}$ and $\balp \in \bR^d$ 
the $C^\ell$-\emph{smooth discrepancy} 
\[
D_{\bomega}(\balp;N) := 
\sup_\fB
\vert 
D_{\bomega}(\fB; N)
\vert,
\]
where the supremum is taken over all boxes 
$\fB$ as in \eqref{def: box B}. 
More generally, one can define the  $C^\ell$-smooth discrepancy 
of a sequence $\bx_1, \bx_2, \ldots$ in $\bR^d$ by taking
\[
D_{\bomega}(\fB; N) := 
\sum_{a_1,\ldots,a_d,n \in \bZ}
\ome_k \left( \frac{n}{N} \right) 
\prod_{i \le d}
\ome_i \left(
\frac{a_i + x_{n,i} - \gam_i}{\rho_i}\right) -
\fC(\bomega) \vol(\fB; N).
\]

\subsection{A deterministic local-to-global principle}

Fix $k \ge 2$ and 
$\balp \in \bR^{d}$, where $k=d+1$. 
Let 
$
\phi: [1,\infty) \to [1,\infty)
$
be a non-decreasing function. For $x \ge \phi(1)$, 
denote by $L(x) \in [1,x]$ 
the value of $H$ such that
\begin{equation}
\label{def L}
H \phi(H) = x.
\end{equation}
In many interesting cases we have $\phi(L(x)) \asymp \phi(x)$, for example if $\phi$ is a power of a logarithm, but in general it can be smaller. Surprisingly, for an appropriate choice of $\phi$ it transpires that $\phi(L(N))$ is the precise threshold for the smooth discrepancy of a Kronecker sequence.

We directly relate the smooth discrepancy to
multiplicative diophantine approximation.
In our first new result, we assume that $\balp$ is $\phi$-badly approximable in a 
dual multiplicative sense: 
\begin{equation}
\label{def phi}
\| \bm \cdot \balp \| > 
\frac1{H(\bm) \phi(H(\bm))} \qquad
(\bzero \ne \bm \in \bZ^{d}).
\end{equation}

\begin{thm}
\label{thm: main}  
Let $\balp \in \bR^d$. Suppose $\phi$ satisfies \eqref{def phi} and 
\begin{equation}
\label{eq: phi doubling}
\xi := \limsup_{x\rightarrow \infty}
\frac{\phi(2x)}{\phi(x)}<\infty.
\end{equation}
Let $\bomega \in \cG_{k,\ell}$ with 
\begin{equation}
\label{eq smoothness}
\ell > 2 + 2\frac{\log \xi}{\log 2}.
\end{equation}
If $N \ge \phi(1)$, then
$$
D_{\bomega}(\balp; N) \ll \phi(L(N)).
$$
\end{thm}

Our next theorem will tell us that this bound in sharp. In the case \mbox{$d=1$,} we see from Theorem \ref{thm: main} that the smooth discrepancy of the Kronecker sequence of any badly approximable number is bounded, which is in stark contrast to Theorem~\ref{thm: Ehrenfest}. 

In the case $d=2$, which is most relevant to Littlewood's conjecture, Badziahin showed that there exists $(\alp_1, \alp_2) \in \bR^2$ for which we can take 
\[
\phi(N) \asymp \log N \cdot \log \log N,
\]
as we saw in Theorem \ref{BadziahinResult}.
Fregoli and Kleinbock~\cite{FK} recently extended this to arbitrary 
$d \ge 2$, showing  there exists $\balp \in \bR^d$ for which we can take
\[
\phi(N) \asymp
(\log N)^{d-1} \log \log N.
\]
Combining this with Theorem \ref{thm: main}, we see that
\[
D_{\bome}(\balp; N) \ll (\log N)^{d-1} \log \log N
\]
for these Kronecker sequences. When $d = 2$, this is smaller than the classical discrepancy of \textbf{any} sequence, by \eqref{GeneralLower}.

We complement Theorem \ref{thm: main} 
with a matching lower bound.

\begin{thm}
\label{thm: lower}
Let $\bome \in \cG_{k,\ell}$, where $\ell \geq 2$. Suppose 
$\phi: [1, \infty) \to [1, \infty)$ is non-decreasing, and
\begin{equation*}
\label{DualGood}
\| \bm \cdot \balp \| < \frac1{H(\bm) \phi(H(\bm))}
\end{equation*}
for infinitely many $\bm \in \bZ^d$.
Then there exist infinitely many $N \in \bN$ such that
\[
D_\bome(\balp; N) \gg \phi(L(N)).
\]
\end{thm}

\subsection{An approach to Littlewood's conjecture}

Recall that Littlewood's conjecture is the $d=2$ case of the following conjecture.

\begin{conj}[Littlewood's conjecture in $d$ dimensions]
\label{Littlewood} 
Let $\balp \in \bR^d$, where $d \ge 2$. Then
\[
\liminf_{n \to \infty} n \| n \alp_1 \| \cdots \| n \alp_d \| = 0.
\]
\end{conj}

Moreover, this is equivalent to its dual form, see \cite{CSD1955, dM2012}.

\begin{conj}
[Dual form of Littlewood's conjecture in $d$ dimensions]
For $d \ge 2$, any $\balp \in \bR^d$ satisfies
\[
\inf_{\bzero \ne \bm \in \bZ^d} H(\bm)
\| \bm \cdot \balp \| = 0.
\]
\end{conj}

Notice that constant functions $\phi$ trivially satisfy 
\eqref{eq: phi doubling} with $\xi=1$.
Thus, if Littlewood's conjecture fails in $d$ dimensions, 
then we can take $\phi$ bounded in Theorem \ref{thm: main},
obtaining a $d$-dimensional Kronecker sequence with bounded 
$C^3$-smooth discrepancy.
Conversely, suppose now Littlewood's conjecture holds in $d$ dimensions, 
then we can always take $\phi$ unbounded in Theorem \ref{thm: lower}, 
to find that any $d$-dimensional Kronecker sequence 
has unbounded $C^3$-smooth discrepancy. Let us now summarise.

\begin{thm}
\label{equiv}
Let $d \ge 2$ be an integer. Then Littlewood's conjecture holds in 
$d$ dimensions if and only if any $d$-dimensional Kronecker sequence 
has unbounded $C^3$-smooth discrepancy.
\end{thm}
\begin{remark}[Littlewood's conjecture and smoothness]
Thus, Littlewood's conjecture is asking whether the $C^3$-smooth counterpart of the van Aardenne-Ehrenfest theorem is true. 
This a rare instance of a smooth problem being more demanding
than its non-smoothed version, see Question \ref{MainOpenQuestion}.
\end{remark}

\subsection{Smooth discrepancy of unimodular lattices}

We prove Theorems \ref{thm: main} and \ref{thm: lower}
in the more general setting of lattices, which we now describe. 
A \emph{lattice} $\Lambda$ in $\bR^k$ is a discrete, full-rank, additive subgroup of 
$\bR^k$. It is well known that $\Lambda$ can be written as
$\Lambda = A \bZ^k$, for some invertible matrix 
$A\in \bR^{k\times k}$ whose columns are called 
a \emph{basis} for $\Lambda$, and that
$\det(\Lambda) := \vert \det (A)\vert$ is well defined. We say that $\Lambda$ is
\emph{unimodular} if $\det (\Lambda) = 1$, and let 
$$
\sL_k= \{\Lam\subset \bR^k: \Lambda \mathrm{\,is\,a\,unimodular\,lattice}\}
\cong \mathrm{SL}_{k}(\bR)/\mathrm{SL}_{k}(\bZ).
$$
For $\bgam \in \bR^d$
and $\brho \in (0,1/2)^d$, let $\fB$ be as in \eqref{def: box B}.
For $\Lambda \in \sL_k$, define
\[
D_{\bomega}^{\Lambda}(\fB; N) = 
\sum_{(\blam,\lambda_k) \in \Lambda}
\ome_k \left(\frac{\lam_k}{N}\right) 
\prod_{i \le d}
\ome_i \left(
\frac{\lam_i - \gam_i}{\rho_i}\right) -
\fC(\bomega) \vol(\fB; N).
\]
We associate to $\bome \in \cG_{k,\ell}$ and $\Lambda \in \sL_k$
the \emph{smooth lattice discrepancy} 
\[
D_{\bomega}(\Lambda;N) := 
\sup_\fB
\vert 
D_{\bomega}^{\Lambda}(\fB; N)
\vert,
\]
where the supremum is taken over all boxes 
$\fB$ as in \eqref{def: box B}. 
Given $\Lam = A \bZ^k \in \sL_k$,
we denote its the dual lattice by 
$$
\Lam^{*}=(A^{-1})^{T} \bZ^k.
$$

We establish the following upper bound.

\begin{thm}
\label{thm: main general lattice} 
Let $\bomega \in \cG_{k,\ell}$ and  $\Lam \in \sL_k$.
Suppose 
$\phi: [1, \infty) \to [1, \infty)$ is a non-decreasing function satisfying
\begin{equation}
\label{eq: lattice dio lower bound}
\vert H(\blam) \lambda_k \vert > \frac{1}{\phi(H(\blam))}
\end{equation}
for any non-zero $(\blam,\lam_k)\in \Lambda^{*}$, 
as well as \eqref{eq: phi doubling}. 
If the function $L$ is given by \eqref{def L}
and $\ell$ satisfies \eqref{eq smoothness}, then
$
D_{\bomega}(\Lam; N) \ll \phi(L(N)).
$
\end{thm}

This upper bound is optimal, as the next results shows.

\begin{thm}
\label{thm: lower general lattice}
Let $\bome \in \cG_{k,\ell}$ with $\ell \geq 2$, and let $\Lam \in \cL_k$. Let 
$\phi: [1, \infty) \to [1, \infty)$ be a non-decreasing function such that
\begin{equation}
\label{DualGood lattice}
\vert H(\blam) \lambda_k \vert < \frac{1}{\phi(H(\blam))}
\end{equation}
holds for infinitely many $(\blam,\lam_k) \in \Lambda^{*}$.
Then there exist infinitely many $N \in \bN$ such that
$
D_\bome(\Lam; N) \gg \phi(L(N)),
$
where the function $L$ is as in \eqref{def L}. 
\end{thm}

\begin{remark}
It is well known that there are $\Lambda \in \sL_k$ such that 
\eqref{eq: lattice dio lower bound} holds with $\phi$ constant.
Here is a construction. Let $\beta\in \bR$ 
be an algebraic number of degree $k$ such that each of its conjugates is real. Then $F=\bQ(\beta)$ is a totally real number field.
Let $\cO_F$ be its ring of integers, and let $\sigma_1,\ldots,\sigma_k$
be the canonical embeddings $F \rightarrow \bR$.
The Minkowski embedding $\cM:\cO_F \rightarrow\bR^k$ 
defined by 
$$
\cM(x)=(\sigma_1(x),\ldots,\sigma_k(x))
$$ 
produces a lattice $\Lambda_{\cO_F} = \cM(\cO_F)$ in $\bR^k$.
For any $x\in \cO_F$, the point $\cM(x)$  
has the property that
$\sigma_1(x) \cdots \sigma_k(x)= N_{F/\bQ}(x)\in \bZ$,
where $N_{F/\bQ}(x)$ is the field norm of $x$, so if $x \ne 0$ then $ \vert \sigma_1(x) \cdots \sigma_k(x)\vert \geq 1$.
Rescaling $\Lambda_{\cO_F}^*$ gives rise to a unimodular lattice $\Lambda$ 
with $ \phi$ in \eqref{eq: lattice dio lower bound} being constant.

This example shows that the smooth lattice discrepancy can be bounded. We do not know whether this can occur for the smooth discrepancy of a sequence in $d \ge 2$ dimensions, but we know from Theorem \ref{thm: Ehrenfest}
that the classical discrepancy of any sequence must be unbounded. See Question \ref{MainOpenQuestion}.
\end{remark}

\subsubsection{Deduction of Theorems \ref{thm: main} and 
\ref{thm: lower} from Theorems \ref{thm: main general lattice} and \ref{thm: lower general lattice}}

Let 
$$
\mathbb{I}_{k-1}=\mathrm{diag}(1,\ldots,1)
$$
be the identity matrix in $k-1$ dimensions.
To study the diophantine 
properties of the column vector $\balp\in\bR^{k-1}$, 
we consider the \emph{Dani lattice}
\begin{equation*}
\label{def Dani lattice}
\Lam_\balp :=
\begin{pmatrix}
\mathbb{I}_{k-1} & \balp\\
\bzero^T & 1 \\
\end{pmatrix} 
\bZ^k.
\end{equation*}
To proceed, we notice that 
$$D_{\bomega}(\balp; N)=D_{\bomega}(\Lam_{\balp}; N)
$$
and
\begin{equation*}
\label{eq dual Dani}
\Lam_\balp^{*} =
\begin{pmatrix}
\mathbb{I}_{k-1} & \bzero \\
-\balp^{T} & 1 \\
\end{pmatrix} 
\bZ^k.
\end{equation*}
Moreover, we see that \eqref{def phi} and \eqref{eq: lattice dio lower bound} 
are compatible. 
Thus, Theorem \ref{thm: main} follows
from Theorem \ref{thm: main general lattice}.
In the same way, Theorem \ref{thm: lower} is implied by 
Theorem \ref{thm: lower general lattice}.

\subsubsection{On a special case of Margulis's conjecture}

Cassels and Swinnerton-Dyer \cite{CSD1955} derived Littlewood's conjecture from the $n=3$ case of the following statement.

\begin{conj}
[Littlewood's conjecture for products of linear forms]
\label{Marg1}
Let $f$ be a product of $k$ linear forms in $\bR^k$, where $k \ge 3$, and assume that $f$ is not proportional to a multiple of a polynomial with integer coefficients. Then
\[
\inf \{ |f(\bx)|: \bzero \ne \bx \in \bZ^k \} = 0.
\]
\end{conj}

This is equivalent to the following special case of a famous problem in homogeneous dynamics, namely Margulis's conjecture \cite[Conjecture 9]{Mar2000}, as Margulis notes in that reference.

\begin{conj}
[Special case of Margulis's conjecture]
\label{Marg2}
Let $D$ be the set of diagonal matrices in $\mathrm{SL}_k(\bR)$, where $k \ge 3$.
If $z \in \mathrm{SL}_k(\bR)/ \mathrm{SL}_k(\bZ)$ and $Dz$ has compact closure, then $Dz$ is closed.
\end{conj}

We present the following refinement of Conjectures \ref{Marg1} and \ref{Marg2}.

\begin{conj}
\label{Marg3}
Let $\Lam = A \bZ^k \in \sL_k$, where $A \in \mathrm{SL}_k(\bR)$ and $k \ge 3$. Let $f$ be the product of the linear forms defined by the rows of $(A^{-1})^T$, and assume that $f$ is not proportional to a polynomial with integer coefficients. Let $\phi: [1, \infty) \to [1,\infty)$ be a non-decreasing function satisfying \eqref{eq: phi doubling} and \eqref{eq: lattice dio lower bound}.
Let $\bome \in \cG_{k,\ell}$, where $\ell$ satisfies \eqref{eq smoothness}. Then $D_\bome(\Lam; N)$ is an unbounded function of $N$.
\end{conj}

We deduce Conjecture \ref{Marg1} from Conjecture \ref{Marg3} as follows. Let $ \eps > 0$. Conjecture \ref{Marg3} tells us that $D_\bome(\Lam;N)$ is unbounded. Then, by Theorem \ref{thm: main general lattice}, there exists $(\blam, \lam_k) \in \Lam^* \setminus \{ \bzero \}$ such that
\[
|\lam_1 \cdots \lam_k| \le |H(\blam) \lam_k| < \eps.
\]
The deduction is complete since $\lam_1 \cdots \lam_k = f(\bx)$ for some non-zero $\bx \in \bZ^k$.

\subsection*{Organisation}

In \S \ref{prep}, we use the geometry of numbers to provide a sharp cardinality bound for Bohr sets. We will prove Theorems \ref{thm: main general lattice} and \ref{thm: lower general lattice} in Sections \ref{MainProof} and \ref{LowerProof}, respectively, before offering some concluding remarks in \S \ref{ConcludingRemarks}.

\section{The geometry of numbers and Bohr sets}
\label{prep}

We start by recalling some classical theory from the geometry of numbers. We then finish the section by using it to prove a sharp cardinality bound for Bohr sets.

\subsection{The geometry of numbers}
Let $\Gamma$ be a lattice in $\bR^k$.
The set 
$$
\{ t_1 \bv^{(1)} + \cdots + t_k \bv^{(k)}: 
0 \le t_1, \ldots, t_k < 1 \}
$$
is a \emph{fundamental region} 
for $\Gamma$ if $\bv^{(1)}, \ldots, \bv^{(k)}$ is a basis for $\Gamma$.
The volume of any fundamental region of $\Gamma$ is equal to 
$\det(\Gamma)$. 
An \emph{admissible body} is
a bounded, convex set $\cC \subset \bR^k$ of positive volume that is symmetric about the origin, 
i.e. if $\bc \in \cC$ then 
$-\bc\in \cC$.
For $i=1,2,\ldots,k$, we denote the $i^{th}$ successive minimum 
of $\Gamma$ with respect to $\cC$ by 
$$
\lambda_i(\Gamma,\cC):= \min\{r>0: 
r \cC \cap \Gamma \mathrm{\,contains\,}i
\mathrm{\, linearly\,independent\,vectors} \}.
$$
We will use the following standard estimate, which is usually attributed to Blichfeldt, see \cite[Corollary 2.10]{Wid2012}.

\begin{lem}
\label{le simple Blichfeld}
Let $\cC \subset \bR^k$ be an admissible body, and let $\Lambda$ be a unimodular lattice in 
$\bR^k$ for which $\lambda_{k}(\Lambda,\cC)\leq 1$. 
Then 
$
\#(\cC\cap \Lambda)
\ll_k \mathrm{vol}(\cC).
$
\end{lem}

To utilise the previous lemma, we recall that the dual body of $\cC$ is
$$
\cC^{*}:=
\{ \bx \in \bR^k: 
\langle \bx, \bc \rangle \le 1
\: (\bc\in \cC)
\}.
$$
The relevance is demonstrated by the 
following case 
of Mahler's relations \cite[Theorem 23.2]{Gru2007}.

\begin{thm}
\label{thm Mahler rel}
There exists $c(k)>1$ such that 
$$
1\leq 
\lambda_{k}(\Gamma,\cC) 
\lambda_{1}(\Gamma^{*},\cC^{*}) 
\leq c(k)
$$
holds for any lattice $\Gamma$ in $\bR^k$
and any admissible body $\cC \subset \bR^k$.
\end{thm}

The situation where the admissible body is a 
symmetric, axis-parallel hyper-rectangle
\begin{equation}
\label{def box}
\cP_{\bs}:= 
[-s_1,s_1]
\times\cdots\times 
[-s_k,s_k],
\quad \mathrm{where\,all\,}s_i>0,
\end{equation}
is important for us. The next lemma will facilitate
applications of Theorem \ref{thm Mahler rel}.

\begin{lem}
\label{le dual cube inclusion}
The parallelepiped $\cP_{\bs}$,
given by \eqref{def box}, satisfies 
$\cP_{\bs}^* \subseteq \cP_{\bs^{-1}}$.
\end{lem}

\begin{proof}
Put $\cQ = [-1,1]^k$ and
$S =\mathrm{diag}(s_1,\ldots,s_k)$, so that $\cP_{\bs} = S\cQ$.
Then
\begin{align*}
\cP_{\bs}^*  
&= \{ \bm \in \bR^k: 
 \langle S \bm, \by \rangle \le 1
\,\, (\by \in \cQ)  \} \\
&= S^{-1}
\{ \bm \in \bR^k: 
 \langle  \bm, \by \rangle \le 1
\,\, (\by \in \cQ) \} 
= S^{-1} \cQ^*.
\end{align*}
As $\cQ^*\subseteq \cQ$,
the desired inclusion follows.
\end{proof}

\subsection{A sharp size estimate for Bohr sets}

We used Bohr sets for multiplicative diophantine approximation in \cite{Cho2018, CT2019, CT2023, CT2024}. Presently, we require size estimates. The sets that we consider generalise the lifted Bohr sets in \cite{CT2019}.

Let $\brho \in (0,1/2)^d$ and $\bgam \in \bR^k$. A Bohr set 
\[
B := B_\Lam^\bgam(N;\brho)
:= (\Lambda-\bgam)\cap\cP_{(\brho,N)}
\]
is \emph{homogeneous} if $\bgam = \bzero$. Here
$
\Lam - \bgam = \{ \blam - \bgam: \blam \in \Lam \}.
$
For $\Lam \in \cL_k$, the Lipschitz principle suggests that 
$B$ has roughly 
\[
\vol(B) := 2^k \rho_1 \cdots \rho_{d} N
\]
many elements.
The next lemma states that this heuristic delivers 
a correct upper-bound, up to a multiplicative constant,
as soon as $\vol(B)$ is sufficiently large.
\begin{lem} 
[Homogeneous size bound]
\label{outer}
Let $\phi: [1,\infty) \to [1,\infty)$ be a non-decreasing function, 
and suppose $\Lam \in \sL_k$ satisfies \eqref{eq: lattice dio lower bound}. 
Then any Bohr set
$B = B_\Lam^\bzero(N; \brho)$ with
\begin{equation} 
\label{crit}
\vol(B) \ge \phi(L(N)) 
\end{equation}
satisfies 
$\# B \ll \vol(B)$.
\end{lem}

\begin{remark}
Using a standard trick, we will see at the end of \S \ref{SmallScalesPart} that this result extends to inhomogeneous Bohr sets, at the cost of replacing $\phi(L(N))$ by $\phi(L(2N))$ in the assumption \eqref{crit}.
\end{remark}

We introduce some notation before proceeding with the proof.
For $\bx,\by\in \bR^d$, we denote component-wise inequality
in absolute value of $\bx,\by$ via
$$
\bx \preccurlyeq \by \qquad
\Leftrightarrow \qquad
\vert x_i \vert \leq \vert y_i\vert  
\quad (1\leq i\leq d).
$$
Moreover, we abbreviate
$$
\brho^{-1}:=(\rho_1^{-1},
\ldots, \rho_{d}^{-1}).
$$

\begin{proof} 
[Proof of Lemma \ref{outer}]
Put $\bs = (\brho,N)$.
Suppose we knew that
\begin{equation}
\label{eq: last min bounded}
\lambda_k(\Lambda,\cP_{\bs})\ll_k 1.
\end{equation}
Then, after enlarging $\cP_{\bs}$
by a constant $1<K\ll_k 1$, 
Lemma \ref{le simple Blichfeld}
would give
$ \# (\Lambda \cap K \cP_{\bs})
\ll \mathrm{vol}(\cP_{\bs})=\mathrm{vol}(B)$,
and the proof would be complete.

In view of Theorem \ref{thm Mahler rel},
the estimate \eqref{eq: last min bounded}
will follow if we can show that
$\lambda_1(\Lambda^{*},\cP_{\bs}^{*})\gg_k 1$.
Lemma \ref{le dual cube inclusion} implies that
$\lambda_1(\Lambda^{*},\cP_{\bs}^{*})
\geq \lambda_1(\Lambda^{*},\cP_{\bs^{-1}}).$
We will prove \emph{a fortiori} that
$\lambda_1(\Lambda^{*}, \cP_{\bs^{-1}}) > 1/4$,
by demonstrating that
\[
\Lambda^{*} \cap \frac14 \cP_{\bs^{-1}}= \{ \bzero \}.
\]

By the diophantine assumption \eqref{eq: lattice dio lower bound},
any non-zero vector $ (\blam,\lambda_k)\in \Lambda^*$ with 
$\blam\preccurlyeq \frac14 \brho^{-1}$ satisfies
$$
\vert \lambda_k \vert  > 
\frac1{H(\blam) \phi(H(\blam))} \geq 
\frac{2^k}{H(\brho^{-1}) \phi(2^{-k} H(\brho^{-1}))}
= 
\frac{\mathrm{vol}(B)}
{N\phi(2^{-k} H(\brho^{-1}))}.
$$
By \eqref{crit}, we have
$H(\brho^{-1}) \leq 2^k N / \phi(L(N)) = 2^k L(N)$,
whence
\[
\vert \lambda_k \vert > 
\frac{\mathrm{vol}(B)}
{N\phi(L(N))} \geq \frac1N
= \frac1{s_k}.
\]
Therefore 
$\Lambda^{*} \cap \frac14 \cP_{\bs^{-1}} = 
\{ \bzero \}$, 
as required.
\end{proof}

\section{Proof of Theorem \ref{thm: main general lattice}}
\label{MainProof}

We split the derivation of Theorem 
\ref{thm: main general lattice} into two cases.
Boxes $\fB$ of small volume are estimated 
in Lemma \ref{lem discr small box}, and 
Proposition \ref{prop discr large box}
deals with those of large volume.

\subsection{Discrepancy at small scales: inflation
and Bohr sets}
\label{SmallScalesPart}

\begin{lem}
\label{lem discr small box}
If $\vol(\fB; N) \le \phi(L(N))$, then 
$D_{\bomega}^{\Lambda}(\fB; N) \ll \phi(L(N))$.
\end{lem}

\begin{proof} Recall $\fC(\bome)$ from \eqref{def expect weight}. 
Since $\omega_i\geq 0$ for each $i$,
$$
D_{\bomega}^{\Lambda}(\fB; N) \geq - \fC(\bome)
\vol(\fB; N) \ge - \fC(\bome) \phi(L(N)).
$$

To bound $D_{\bomega}^{\Lambda}(\fB; N)$ 
from above, we write $\fB = (\bgam; \brho)$. We may assume that 
$L(4N) > 2^{k+1}$, so that $\phi(L(4N)) = 4N/L(4N) < 2^{-d} N$. 
We inflate $\fB = (\bgam; \brho)$ to $\fB'=(\bgam; \bdel)$,
where $\bdel=(\del_1,\ldots,\del_{k-1})$
is chosen so that 
\[
\rho_i \leq \del_i < 1/2 \quad (1 \le i \leq k-1),
\qquad
\vol(\fB'; N) = \phi(L(4N)).
\]
We know from \eqref{eq omega supported dyadically} that $\ome_i(x) \ll \mathds{1}_{[-2,2]}(x)$ for each $i$, whence
\begin{align*} 
D_{\bomega}^\Lam(\fB; N) &\ll \phi(L(N))
+ \# B,
\end{align*}
where $B = (\Lambda - (\bgam,0)) \cap \cP_{2(\bdel,N)}$.

We now cover $B$ by at most $2^k$ many Bohr sets
$B' = B_\Lam^{\btau}(2N; \boldsymbol \eta)$, where 
\[
\frac{\del_i}2 \le \eta_i < \frac14
\qquad (1 \le i \le d).
\]
If $B' \neq \emptyset$, then 
we fix $\bn_0 \in B'$ and write the elements $\bn \in B'$
as $\bn = \bm + \bn_0$. By the triangle inequality, we have $\bm \in B_\balp^\bzero(4N;2 \boldsymbol \eta)$.
Therefore
$\# B' \le 
\# B_\Lam^\bzero(4N;2 \boldsymbol \eta)$.
Lemma \ref{outer} yields
$
\# B_\balp^\bzero(4N;2 \boldsymbol \eta) \ll \phi(L(4N)),
$
so $D_{\bomega}(\fB; N) \ll \phi(L(4N))$. Finally, we note that $\phi \circ L$ is doubling:
\[
\phi(L(4N)) = \frac{4N}{L(4N)} \le \frac{4N}{L(N)} = 4\phi(L(N)).
\qedhere
\]
\end{proof}

\subsection{Large scales:
Poisson summation and a gap principle}
\label{LargeScales}

Throughout this subsection, we work with the remaining case 
of large test boxes $\fB$, that is,
\begin{equation}
\label{eq large boxes}
\vol(\fB; N)=\frac{N}{H(\brho^{-1})}  > \phi(L(N)).
\end{equation}
Our goal is to establish the following bound.

\begin{prop}
\label{prop discr large box}
If $\vol(\fB; N) > \phi(L(N))$, then $D_{\bomega}^\Lambda(\fB;N) \ll \phi(L(N))$.
\end{prop}

We begin by passing to Fourier space.

\begin{lem}
\label{le discr small box}
Define 
$$
\cL = \Lambda^*\setminus\{\bzero\}
$$
and
\begin{align}
D_{\bomega}^+(\fB;N)
\label{def trunc contr}
= \fC(\bomega) \vol(\fB; N) 
\sum_{\blam \in \cL} 
\widehat \ome_k
(N \lambda_k)
\prod_{i \le d} 
\widehat{\ome}_i(\rho_i \lambda_i).
\end{align}
Then
\begin{equation}
\label{UpperTruncate}
|D_{\bomega}^\Lam(\fB;N)| \le D_{\bomega}^+(\fB;N).
\end{equation}
Moreover, if $\bgam = \bzero$ then
\begin{equation}
\label{LowerTruncate}
D_{\bomega}^\Lam(\fB;N)= D_{\bomega}^+(\fB;N).
\end{equation}
\end{lem}

\begin{proof}
Let $\rho_k = N$ and 
$\gam_k =0$.
Then
\[
D_{\bomega}^\Lam (\fB;N) + \fC(\bomega) \vol(\fB; N) = 
\sum_{\boldsymbol \ell \in \Lam} 
g(\boldsymbol \ell), \quad 
\text{where}
\quad
g(\boldsymbol \ell) =
\prod_{i \le k} 
\ome_i\left(\frac{\ell_i - \gam_i}
{\rho_i} \right).
\]
Note that 
$$
\widehat{g}(\bx)=
e(\bx \cdot (\bgam,0))
\prod_{i \le k} 
(\rho_i \widehat{\ome}_i(\rho_i x_i))
=
\vol(\fB; N)
e(\bx\cdot (\bgam,0))
\prod_{i \le k} 
\widehat{\ome}_i(\rho_i x_i).
$$
Poisson summation yields
$$
\sum_{\boldsymbol \ell \in \Lam} 
g(\boldsymbol \ell)
=
\sum_{(\blam,\lambda_k) \in \Lam^{*}} 
\widehat{g}((\blam,\lambda_k) )
= 
\vol(\fB; N)
\sum_{(\blam,\lambda_k) \in \Lam^{*}} 
e(\blam \cdot \bgam)
\prod_{i \le k} 
\widehat{\ome}_i (\rho_i \lambda_i).
$$
Upon accounting for the contribution of $(\blam, \lam_k) = \bzero$, and using the triangle inequality, 
we obtain \eqref{UpperTruncate}. If $\bgam = \bzero$, then the summands are non-negative, so we have equality in that final step.
\end{proof}

In the analysis of $D_{\bomega}^+(\fB;N)$, a special role is played by the regime
$$
\cU := \{ (\blam,\lam_k)\in \Lam^*: 
\bzero \neq \blam \preccurlyeq \brho^{-1} \text{ and }
\vert \lambda_k\vert \leq N^{-1}
\} \subseteq \cL
$$
coming from the uncertainty principle, since here none of the 
$\widehat{\ome}_i(\rho_i \lambda_i)$ decay.
One can morally replace each $ \widehat{\ome}_i(\rho_i \lambda_i)$ 
by $1$ in this range. Bearing the factor of 
$\vol(\fB;N)$ in mind, this could conceivably 
lead to an unacceptably large contribution from just one term.
To rule this out, we show that $\cU = \emptyset$.
To see this, observe that if $(\blam,\lambda_k) \in \cU$ then
$$
\vert \lambda_k \vert >
\frac{1}{H(\blam) \phi(H(\blam))}
\ge \frac{1}{H(\brho^{-1}) \phi(H(\brho^{-1}))}
= \frac{\vol(\fB;N)}
{N\phi(H(\brho^{-1}))}.
$$
By \eqref{eq large boxes}, 
we have $H(\brho^{-1}) \le N/\phi(L(N)) = L(N)$.
Consequently
$$
\vert \lambda_k \vert  >
\frac{\phi(L(N))} {N\phi(L(N))}
= \frac{1}{N},
$$
whence $\cU=\emptyset$.

Beyond $\cU$, we will make use of the decay of the 
$\widehat{\omega}_i$.
Indeed, partial integration implies that
\begin{equation}
\label{eq rapid decay}
\widehat{\omega}_i(\xi)
\ll_\ell (1+\vert \xi \vert)^{-\ell}
\qquad (\xi\in \bR, \quad 1 \le i \le k).
\end{equation}
To use the decay efficiently, 
we employ the scale vectors 
\begin{equation*}
\label{def pseudo dyadic}
\bfN((\bd,t_k)):=
(\rho_1^{-1} \breve{t}_1 ,
\ldots,\rho_{k-1}^{-1} \breve{t}_{k-1}, N^{-1} \breve{t}_k),
\quad \mathrm{where}\quad \breve{t} = 2^{t}-1.
\end{equation*}
Furthermore, put 
$\mathbf{1} = (1,\ldots,1) \in \bR^{k}$ and
\begin{equation*}
\label{def enlarged uncert}
\cU(\btau) = 
\{ \boldsymbol \ell \in \cL: 
\bfN(\btau) \preccurlyeq
\boldsymbol \ell \preccurlyeq \bfN(\btau + \bone)
\}.
\end{equation*}
Note that $\cU(\bzero)=\cU$ and
\begin{equation}
\label{eq covering by pseudo dyadics}
\cL\setminus \cU = \bigcup_{\btau \in \fD} \cU(\btau),
\quad \mathrm{where} \quad 
\fD=\bZ^{k}_{\geq 0}\setminus\{\bzero\}.
\end{equation}

We exploit the decay of the $\widehat{\omega}_i$,
which dampens the contribution of lattice points
outside of the uncertainty range $\cU$. 
To track the extent to which $\cU((\bd,t_k))$ exceeds $\cU$, the quantity \begin{equation*} \label{def quotient of dyadics}
Q(\bt) := 
2^{t_1+\cdots+t_d}
\end{equation*}  
is helpful. 
By \eqref{eq rapid decay}, if
$(\blam, \lam_k) \in \cU((\bd,t_k))$ then
\begin{align}
\label{eq decay in dyadic range}
\prod_{i \le d} \widehat \ome_i (\rho_i \lam_i) &\ll
H((\rho_1 \lam_1, \ldots, \rho_{d} \lam_{d}
))^{-\ell} \ll Q(\bt)^{-\ell}.
\end{align}
Combining \eqref{def trunc contr}, 
\eqref{eq covering by pseudo dyadics} 
and \eqref{eq decay in dyadic range}
delivers
\begin{equation}
\label{eq estimate via dyadic cover}
D_{\bomega}^+(\fB;N)
\ll 
\sum_{\btau \in \fD}
X(\btau),
\end{equation}
where
\begin{equation*}
X((\bt,t_k)) = \frac{\vol(\fB;N)}{Q(\bt)^{\ell}} 
\sum_{(\blam,\lam_k) \in \cU((\bt,t_k))}  \widehat \ome_k (N \lam_k).
\end{equation*}

To proceed, we bound each $X((\bt, t_k))$. 
For ease of notation, we write
$$
\bfN((\bd,t_k) + \bone) = (\mathbf N_\bt, N_k) = (\mathbf N, N_k), \qquad
Q = Q(\bt).
$$
Note that
\begin{equation}
\label{Qsize}
Q \asymp \frac{H(\bfN)}{H(\brho^{-1})}.
\end{equation}

\begin{lem}
\label{le app regime outside uncert}
If $\btau =(\bt,t_k) \in \fT$, then
\[
X(\btau) \ll 
(2^{t_k}Q)^{2-\ell} \frac{\phi(2^d H(\bfN))^2}{\phi(L(N))},
\]
where the implied constant does not depend on $\btau$.
\end{lem}

\begin{proof}
Consider the function
$\pi_k: \cU(\btau) \rightarrow \bR$
defined by $(\blam,\lambda_k) \mapsto \lambda_k$.
We claim that $\pi_k$
is injective and, moreover, 
any two distinct elements of $\pi_k(\cU(\btau))$
are spaced apart by at least a constant times
\[
u(\btau) := \frac1{H(\bfN) \phi(2^d H(\bfN))}.
\]
Let $\boldsymbol \ell = (\blam, \lam_k)$ and 
$\boldsymbol \ell' = (\blam', \lam_k')$ 
be distinct elements of $\cU(\btau)$.
By \eqref{eq: lattice dio lower bound}
and $H(\blam - \blam')
\le 2^d H(\bfN)$,
we have
$$
\vert \pi_k(\boldsymbol \ell) - \pi_k(\boldsymbol \ell') \vert
\geq \frac{1}
{2^d H(\bfN)
\phi(2^d H(\bfN))} \gg u(\btau),
$$
confirming the claim. 

Next, we enumerate the elements $w_i^+$ 
of 
$$
\cW^+ := \{N \pi_k(\boldsymbol \ell) : 
\boldsymbol \ell \in \cU(\btau), \: \pi_k(\boldsymbol \ell) > 0 \}
$$ 
so that 
$w_i^+ < w_j^+$ if $i<j$. 
Observe that
$w_i^+\gg N u(\btau)$ for any $w_i^+\in \cW^+ $.
Thus, by \eqref{eq large boxes} and \eqref{Qsize},
\[
w_i^+ \gg i N u(\btau) 
\geq i \phi(L(N)) H(\brho^{-1}) u(\btau)
\gg i \frac{\phi(L(N))}
{Q\phi(2^d H(\bfN))}.
\]
If $t_k\geq 1$, then each $w_i^+ \gg 2^{t_k}$ for all $i$. Therefore
\[
(1 + w_i^+)^{\ell} \gg 2^{(\ell - 2) t_k} iNu(\btau) \cdot i\frac{\phi(L(N))}{Q\phi(2^d H(\bfN))}
\quad (1\leq i \leq \#\cW^{+}),
\]
uniformly over $t_k\geq 0$. Similarly, we enumerate
\[
\cW^- := \{N |\pi_k(\boldsymbol \ell)| :
\boldsymbol \ell \in \cU(\btau), \: 
\pi_k(\boldsymbol \ell) < 0 \} 
= \{ w_1^- < w_2^- < \cdots < w_{\# \cW^-}^- \},
\]
and see that
\[
(1 + w_i^-)^{\ell} \gg 2^{(\ell - 2)t_k} iNu(\btau) 
\cdot i\frac{\phi(L(N))}{Q\phi(2^d H(\bfN))}
\quad (1\leq i \leq \#\cW^{-}).
\]
Since $\widehat{\ome}_k(x) \ll (1+\vert x \vert^\ell)^{-1}$,
we conclude that
\begin{align*}   
X(\btau) = \frac{NQ^{-\ell}}{H(\brho^{-1})} 
\sum_{w\in \cW^{+}\cup \cW^{-} } 
\widehat{\ome}_k(w)
& \ll 
\frac{NQ^{-\ell}}{H(\brho^{-1})} 
\sum_{i\geq 1} 
\frac{Q\phi(2^d H(\bfN))}{2^{(\ell- 2) t_k}i^2 N u(\btau) \phi(L(N))} \\
&\ll 
\frac{Q^{1-\ell}}{H(\brho^{-1})} 
\frac{2^{(2-\ell)t_k}\phi(2^d H(\bfN))}{u(\btau) \phi(L(N))}.
\end{align*}
Inserting the definition of $u(\btau)$, and recalling \eqref{Qsize}, yields
\[
X(\btau) \ll 
\frac{Q^{1-\ell}}{H(\brho^{-1})} 
\frac{2^{(2-\ell)t_k}H(\mathbf N)\phi(2^d H(\mathbf N))^2}
{\phi(L(N))} =
(2^{t_k}Q)^{2-\ell} \frac{\phi(2^d H(\bfN))^2}{\phi(L(N))}.
\]
\end{proof}

\begin{proof}
[Proof of Proposition \ref{prop discr large box}]
By Lemma \ref{le discr small box}, 
it suffices to prove that 
\[
D_{\bomega}^+(\fB;N) \ll \phi(L(N)),
\]
and we may assume that $N$ is large. Using Lemma \ref{le app regime outside uncert}
in the estimate \eqref{eq estimate via dyadic cover} 
gives
$$
D_{\bomega}^+(\fB;N)
\ll 
\sum_{(\bt,t_k) \in \fD}
(2^{t_k} Q(\bd))^{2-\ell} 
\frac{\phi(2^d H(\bfN_\bt))^2}{\phi(L(N))}.
$$
Put $\Vert \bd \Vert_1 = t_1 + \cdots + t_{d}$.
By \eqref{eq large boxes},
$$
2^d H(\bfN_\bt) \le 2^{2d+\Vert \bd \Vert_1} H(\brho^{-1}) 
\leq 2^{2d+\Vert \bd \Vert_1} L(N).
$$
Therefore $\phi(2^d H(\bfN_\bt))\leq 
\phi(2^{2d+\Vert \bd \Vert_1} L(N))$.
Let $\xi$ be as in \eqref{eq: phi doubling}.
As $N$ is large, we now have
\[
\phi(2^d H(\bfN_\bt)) 
\leq (\xi+\varepsilon)^{2d+\Vert \bd \Vert_1} 
\phi(L(N)),
\]
where $\varepsilon>0$ is a small constant.
Taking $Q(\bt) = 2^{\Vert \bd \Vert_1}$ into account produces
\begin{align*}
D_{\bomega}^+(\fB;N)
&\ll 
\sum_{(\bd,t_k) \in \fD}
(2^{t_k + \| \bt \|_1})^{(2-\ell)} 
\frac{(\xi+\varepsilon)^{2\Vert \bd \Vert_1} 
\phi(L(N))^2}{\phi(L(N))} \\
& \ll 
\phi(L(N))
\sum_{\bd \in \bZ_{\geq 0}^d}
\bigg(
\frac{(\xi+\varepsilon)^{2}}{2^{\ell - 2}}
\bigg)^{\Vert \bd \Vert_1}.
\end{align*}
By \eqref{eq smoothness}, we have
$2^{\ell - 2} > (\xi+\varepsilon)^2$, completing the proof.
\end{proof}

In light of Lemma \ref{lem discr small box},
this completes the proof of Theorem \ref{thm: main general lattice}.

\section{Proof of Theorem \ref{thm: lower general lattice}}
\label{LowerProof}

We adopt the infrastructure of \S \ref{LargeScales}. The first idea is to centre the box $\fB$ at $\bgam = \bzero$ so that, after Poisson summation, the summands are all non-negative. The second idea is to use a good approximation to find a single large summand, which informs the choice of length vector $\brho$ for the box $\fB$.

By \eqref{positive}, there exists a constant 
$c \in (0, 1/2)$ such that $\widehat \ome_i(x) \ge c$
for $|x| \le c$  and $1 \le i \le k$.
Let $(\blam,\lam_k) \in \Lam^*$ satisfy \eqref{DualGood lattice}, with $H(\blam)$ large. 
Consider $\brho = \brho(\blam)\in (0,1/2)^d$ 
with components $\rho_i = c / (1 + \vert \lam_i \vert)$, for $i\leq d$. 
We take
$$
N = \lfloor c \phi(H(\blam)) H(\blam) \rfloor,
$$
where $\lfloor \cdot \rfloor$
denotes the floor function, and note that $H(\blam) \ge L(N)$. 
As $H(\blam)$ is large, we have $N \in \bN$ with 
$N \asymp \phi(H(\blam)) H(\blam)$.
By \eqref{LowerTruncate},
$$
D_{\bomega}(\Lambda;N) \ge D_{\bomega}^\Lambda(\fB; N) =
D_{\bomega}^+(\fB; N).
$$
As $\blam \in \cL$ contributes at least $c^k$ to the sum in the definition \eqref{def trunc contr} of $D_\bome^+(\fB; N)$, and all of the summands are non-negative, we have
\begin{equation*}
\label{eq: lower bound term}
D_{\bomega}(\Lam; N) \gg
\vol(\fB;N) \gg \frac{\phi(H(\blam)) H(\blam)}{H(\blam)} 
= \phi(H(\blam)) \ge \phi(L(N)).
\end{equation*}
This completes the proof of Theorem \ref{thm: lower general lattice}.

\section{Concluding Remarks}
\label{ConcludingRemarks}

\subsection{The scales where the smooth discrepancy is large}

After applying Poisson summation, we saw that $D_\bome(\fB; N)$ is maximised when $\fB$ is centred at the origin, i.e. $\bgam = \bzero$. For such boxes $\fB$, this realises $D_\bome(\fB; N)$ as a sum of non-negative terms. 
We saw in \S \ref{LowerProof} that a single large summand arises if the scale $N$ is related to a good approximation. A closer look at the proof of Lemma \ref{le app regime outside uncert} reveals that this is, in fact, the only way in which $D_\bome(\fB; N)$ can be large.

\subsection{Dual Kronecker sequences}

We expect there to be a dual analogue of our local-to-global principle. This should involve the smooth discrepancy of the dual Kronecker sequence $(\bn \cdot \balp \mmod 1)_{\bn \in \bZ^d}$. However, we have not explored this and are not in a position to make precise predictions about it.

\subsection{Smooth lattice point counting in general}

Theorem \ref{thm: main general lattice} is a smooth count of lattice points in boxes that are short in all but one direction. The techniques developed here can be used, with additional ideas, to count smoothly-weighted lattice points in general boxes. This is an interesting question in its own right and has various applications. For instance, one could use such a result to smoothly
count the number of algebraic
integers up to a given height
in a 
number field.
Whilst counting results are available without smoothing~\cite{Wid2012},
the smoothed counting results would have much superior error terms and are just as useful for most purposes.

\subsection{A path towards Littlewood's conjecture.}

We have shown that Littlewood's conjecture holds if and only if the smooth discrepancy of any two-dimensional Kronecker sequence is unbounded. It could be that the smooth discrepancy of \textbf{any} sequence in two or more dimensions is unbounded.

\begin{qn} \label{MainOpenQuestion}
Let $\ell \ge 3$. In $d \ge 2$ dimensions, does there exist a sequence with bounded $C^\ell$-smooth discrepancy?
\end{qn}

\noindent
If the answer is negative, like in the classical setting, 
then Littlewood's conjecture would be solved! 

\subsection*{Funding}
NT was supported by the Austrian Science Fund: 
projects \mbox{J $4464$-N} and I $4406$-N,
and was 
funded by the Deutsche Forschungsgemeinschaft 
(DFG, German Research Foundation) -
Project-ID 491392403 - TRR 358.
NT carried out part of this work 
during the 2024 program on analytic number theory
at Institut Mittag-Leffler in Djursholm, Sweden,
which was supported by grant no. 2021-06594 
of the Swedish Research Council.

\subsection*{Acknowledgements}
 
We thank Andy Pollington for helpful discussions.

\subsection*{Rights}

For the purpose of open access, 
SC has applied a Creative Commons Attribution 
(CC-BY) licence to any 
Author Accepted Manuscript version arising 
from this submission.


\end{document}